\documentclass[12pt]{article}

\usepackage{amsthm,amsmath,amssymb}
\usepackage{graphicx}
\usepackage{caption}
\usepackage[colorlinks=true,citecolor=black,linkcolor=black,urlcolor=blue]{hyperref}

\theoremstyle{plain}
\newtheorem{thm}{Theorem}[section]
\newtheorem{lem}[thm]{Lemma}

\newtheorem{cor}[thm]{Corollary}

\theoremstyle{definition}

\theoremstyle{remark}

\def\e{\varepsilon}

\title{\bf List-coloring graphs on surfaces \\with varying list-sizes}

\author{
 Alice M. Dean\\
\small Department of Mathematics and Computer Science\\[-0.8ex]
\small Skidmore College\\[-0.8ex]
\small Saratoga Springs, NY 12866, USA\\
\small\tt adean@skidmore.edu\\
\and
 Joan P. Hutchinson\\
\small Department of Mathematics, Statistics, and Computer Science\\[-0.8ex]
\small Macalester College\\[-0.8ex]
\small St. Paul, MN 55105, USA\\
\small\tt hutchinson@macalester.edu\\
}

\date{
{\bf In Memory of Herbert S. Wilf, 1931-2012}}

\graphicspath{%
    {converted_graphics/}
    {/}
}
\begin{document}

\maketitle

\begin{abstract}
Let $G$ be a graph embedded on a surface $S_\e$ with Euler genus $\e > 0$, and let $P\subset V(G)$ be a set of vertices mutually at distance at least 4 apart. Suppose all vertices of $G$ have $H(\e)$-lists and the vertices of $P$ are precolored, where $H(\e)=\Big\lfloor\frac{7 + \sqrt{24\e + 1}}{2}\Big\rfloor$ is the Heawood number. We show that the coloring of $P$ extends to a list-coloring of $G$ and that the distance bound of 4 is best possible. Our result provides an answer to an analogous question of Albertson about extending a precoloring of a set of mutually distant vertices in a planar graph to a 5-list-coloring of the graph and generalizes a result of Albertson and Hutchinson to list-coloring extensions on surfaces.

\end{abstract}

\section{Introduction}

For a graph $G$ the distance between vertices $x$ and $y$, denoted $dist(x, y)$, is the number of edges in a shortest $x$-$y$-path in $G$, and we denote by $dist(P)$ the least distance between two vertices of $P$.
In~\cite{albertson98} M. O. Albertson asked if there is a distance $d > 0$ such that every planar graph with a 5-list for each vertex and a set of precolored vertices $P$ with $dist(P) \ge d$ has a list-coloring that is an extension of the precoloring of $P$. In that paper he proved such a result for 5-coloring with $d \ge 4$, answering a question of C. Thomassen. There have been some preliminary answers to Albertson's question in~\cite{ahl, dlm, kawaramohar09}; initially Tuza and Voigt~\cite{tv} showed that $d > 4$. 
Kawarabayashi and Mohar~\cite{kawaramohar09} 
have shown that when $P$ contains $k$ vertices, there is a function $d_k > 0$ that suffices for such list-coloring. Then recently Dvo\v{r}\'{a}k, Lidick\'{y}, Mohar and Postle~\cite{dlmp} have announced a complete solution, answering Albertson's question in the affirmative, independent of the size of $P$. 

Let $S_\e$ denote a surface of Euler genus $\e > 0$. Its Heawood number is given by
	\[H(\e) = \Big\lfloor\frac{7 + \sqrt{24\e + 1}}{2}\Big\rfloor\] 
and gives the best possible bound on the chromatic number of $S_\e$ except for the Klein bottle whose chromatic number is 6. (For all basic chromatic and topological graph theory results, see~\cite{jt, mt}.) In many instances results for list-coloring graphs on surfaces parallel classic results on surface colorings. Early on it was noted that the Heawood number also gives the list-chromatic number for surfaces; see~\cite{jt} for history. Also Dirac's Theorem~\cite{dirac57} has been generalized to list-coloring by B\"{o}hme, Mohar and Stiebitz for most surfaces; the missing case, $\e = 3$, was completed by Kr\'{a}l' and \v{S}krekovski. This result informs and eases much of our work.

\begin{thm}[\cite{bohme99, kral06}]\label{thm:dirac} If $G$ embeds on $S_\e$, $\e > 0$, then $G$ can be $(H(\e)-1)$-list-colored unless $G$ contains $K_{H(\e)}$.
\end{thm}

Analogously to Albertson's question on the plane, we and others (see~\cite{kawaramohar09}) ask related list-coloring questions for surfaces. In this paper we ask if there is a distance $d > 0$ such that every graph on $S_\e$, $\e > 0$, with $H(\e)$-lists on each vertex and a set of precolored vertices $P$ with $dist(P) \ge d$ has a list-coloring that is an extension of the precoloring of $P$. In~\cite{ah02} Albertson and Hutchinson proved the following result; the main result of this paper generalizes this theorem to list-coloring.

\begin{thm}[\cite{ah02}]\label{thm:ah02}
For each $\e > 0$, except possibly for $\e = 3$, if $G$ embeds on a surface of Euler genus $\e$ and if $P$ is a set of precolored vertices with $dist(P)\ge6$, then the precoloring extends to an $H(\e)$-coloring of $G$.
\end{thm}

Others have studied similar extension questions with $k$-lists on vertices for $k \ge 5$. For example, see~\cite{thomassen97}, Thm. 4.4, for $k\ge6$ and~\cite{kawaramohar09}, Thm. 6.1, for $k=5$; however, in both results
 the embedded graphs must satisfy constraints depending on the Euler genus and the number of precolored vertices. 
Our main result is Thm.~\ref{thm:mainresult}, which shows that there is a constant bound on the distance between precolored vertices that ensures list-colorability for all graphs embedded on all surfaces when vertices have $H(\e)$-lists. It improves on Thm.~\ref{thm:ah02} by removing the possible exception for $\e=3$, reducing the distance of the precolored vertices from 6 to 4, and broadening the results to list-coloring.

\begin{thm}\label{thm:mainresult} Let $G$ embed on $S_\e$, $\e > 0$, and let $P \subset V(G)$ be a set of vertices with $dist(P)\ge4$. Then if the vertices of $P$ each have a 1-list and all other vertices have an $H(\e)$-list, $G$ can be list-colored. The distance bound of 4 is best possible.
\end{thm}

When $G$ is embedded on $S_\e$, let the \emph{width}~\cite{width} denote the length of a shortest noncontractible cycle of $G$; this is also known as \emph{edge-width}. For list-coloring we have the following corollary of Thms.~\ref{thm:dirac} and~\ref{thm:mainresult}.

\begin{cor}\label{cor:firstwidth}
If $G$ embeds on $S_\e$, $\e > 0$, with width at least 4, if the vertices of  $P\subset V(G)$ 
have 1-lists and all other vertices have $H(\e)$-lists, then $G$ is list-colorable when $dist(P)\ge3$. The distance bound of 3 is best possible.
\end{cor}

Given that graphs embedded with very large width can be 5-list-colored as proved in~\cite{dkm}, it is straightforward to deduce a 6-list-coloring extension result for such graphs. When $G$ embeds on $S_\e$, $\e > 0$, with width at least $2^{O(\e)}$, if a set of vertices $P$ with $dist(P)\ge3$ have 1-lists and all others have 6-lists, then after the vertices of $P$ are deleted and the color of each $x \in P$ is deleted from the lists of $x$'s neighbors, the remaining graph has 5-lists, large width, and so is list-colorable. Thus $G$ is list-colorable, but only when embedded with large width whose size increases with the Euler genus of the surface.

A consequence of Thomassen's proof of 5-list-colorability of planar graphs \cite{thomassen94} is that if all vertices of a graph in the plane have 5-lists except that the vertices of one face have 3-lists, then the graph can be list-colored. For surfaces, we offer as a related result another corollary of Thm.~\ref{thm:mainresult}.

\begin{cor}\label{cor:twosizes}
If $G$ embeds on $S_\e$, $\e > 0$, and contains a set of faces each pair of which is at distance at least two apart, with all vertices on these faces having $(H(\e)-1)$-lists and all other vertices having $H(\e)$-lists, then $G$ can be list-colored. 
\end{cor}

The paper concludes with related questions.

\section{Background results on surfaces,\\ Euler genus and the Heawood formula}

Let $S_\e$ denote a surface of Euler genus $\e > 0$. If $\e$ is odd, then $S_\e$ is the nonorientable surface with $\e$~crosscaps, but when $\e$ is even, $S_\e$ may be orientable or not. We let {\Large $\tau$} denote the torus, the orientable surface of Euler genus 2, and {\Large $\kappa$} the Klein bottle, the nonorientable surface of Euler genus 2.

The Heawood number $H(\e)$, defined above, gives the largest $n$ for which 
$K_n$ embeds on a surface $S_\e$ of Euler genus $\e$, as well as the chromatic number of $S_\e$, except that $K_6$ is the largest complete graph embedding on {\Large $\kappa$} and 6 is its chromatic number.

The least Euler genus $\e$ for which $K_n$ embeds on $S_\e$ is given by the inverse function
	 \[\e = I(n) = \Big\lceil\frac{(n-3)(n-4)}{6}\Big\rceil.\]

Each $K_n$, $n \ge 5$, of course, has a minimum value of $\e > 0$ for which it embeds on $S_\e$, called the \emph{Euler genus of $K_n$}, but for $\e \ge 2$ more than one surface $S_\e$ may have the same maximum $K_n$ that embeds on it. For example, both $S_5$ and $S_6$ have Heawood number 9 with $K_9$ being the largest complete graph that embeds there. Embedding patterns of $K_{H(\e)}$ depend on the congruence class of $H(\e)$ modulo 3 for $\e \ge 1$. In Table~\ref{tab:heawood}, which gives values of $\e$ and $H(\e)$ for $\e=1, \ldots, 24$, $e$ is the number of edges in $K_{H(\e)}$, $f = 2-\e-v+e$ is the number of faces in a 2-cell embedding of $K_{H(\e)}$ on $S_\e$, and the final column gives the size of the largest possible face when $K_{H(\e)}$ is so embedded. That largest face size is three more than the difference $2e - 3f$.

\begin{table}
\begin{center}
\begin{tabular}{|c|c|c|c|c||c|c|c|c|c|}
\hline
 $\e$&$H(\e)$&$e$&$f$&Largest&$\e$&$H(\e)$&$e$&$f$&Largest\\
 & & & &Face& & & & &Face\\
\hline
1&6&15&10&3&
13&12&66&43&6\\
2&7&21&14&3&
14&12&66&42&9\\
3&7&21&13&6&
15&13&78&52&3\\
4&8&28&18&5&
16&13&78&51&6\\
5&9&36&24&3&
17&13&78&50&9\\
6&9&36&23&6&
18&13&78&49&12\\
7&10&45&30&3&
19&14&91&60&5\\
8&10&45&29&6&
20&14&91&59&8\\
9&10&45&28&9&
21&14&91&58&11\\
10&11&55&36&5&
22&15&105&70&3\\
11&11&55&35&8&
23&15&105&69&6\\
12&12&66&44&3&
24&15&105&68&9\\
\hline
\end{tabular}
 \caption{Embedding parameters for $K_{H(\e)}$ }
 \label{tab:heawood}
\end{center}
\end{table}

For our results we need to know when $K_{H(\e)}$ necessarily has a 2-cell embedding on $S_\e$. When $ K_n$ embeds on $S_\e$, but not on $S_{\e-1}$, then $K_n$ necessarily embeds with a 2-cell embedding. When $K_n$ embeds in addition on $S_{\e+1}$, \ldots, $S_{\e+i}$ with $i > 0$, then it may not have a 2-cell embedding on the latter surfaces. For example, on surfaces $S_1$, {\Large$ \tau$}, $S_4$, and $S_5$, the complete graphs $K_6, K_7, K_8$ and $K_9$ have 2-cell embeddings, respectively, but $K_6, K_7$ and $K_9$ may or may not have 2-cell embeddings on {\Large $\kappa$}, $S_3$ and $S_6$, respectively. 

If $f$ is a face of an embedded graph $G$, let $V(f)$ and $E(f)$ denote the incident vertices and edges of $f$. We say that $V(f) \cup E(f)$ is the \emph{boundary} of $f$ and that the \emph{closure} of $f$ is the union of $f$ and its boundary. 
Each edge of $E(f)$ either lies on another face besides $f$ or it might lie just on $f$. For example, Fig.~\ref{fig:2cell} shows two graphs embedded on the torus, {\Large$ \tau$}. In the first graph, edges 2-3 and 4-7 each border two faces, but edges 3-6 and 8-9 each border only one face. The size $s$ of a face $f$ is determined by counting, with multiplicity, the number of edges on its boundary, and we then call $f$ an $s$-region. In other words, when $s_1$ edges of $E(f)$ lie on another face of $G$ besides $f$ and $s_2$ edges lie only on $f$, then we call $f$ an $s$-region where $s = s_1 + 2s_2$. When $f$ is a 2-cell, $E(f)$ forms a single facial walk $W_f$, and the size of the face equals the length of the facial walk, counting multiplicity of repeated edges. Since an $s$-region $f$ may have repeated edges and repeated vertices, we indicate $|V(f)| = t$ by calling $f$ also a $t$-vertex-region where $t \le s$. Hence the shaded region in the first graph in 
Fig.~\ref{fig:2cell} is a 13-region and a 9-vertex-region, since two edges and four vertices are repeated; the shaded region in the second graph, with no repeated vertices or edges, is a 13-region and a 13-vertex-region.

\begin{figure}[htbp] 
\centerline{\includegraphics[bb=5 14 564 256,width=.7\textwidth,keepaspectratio]{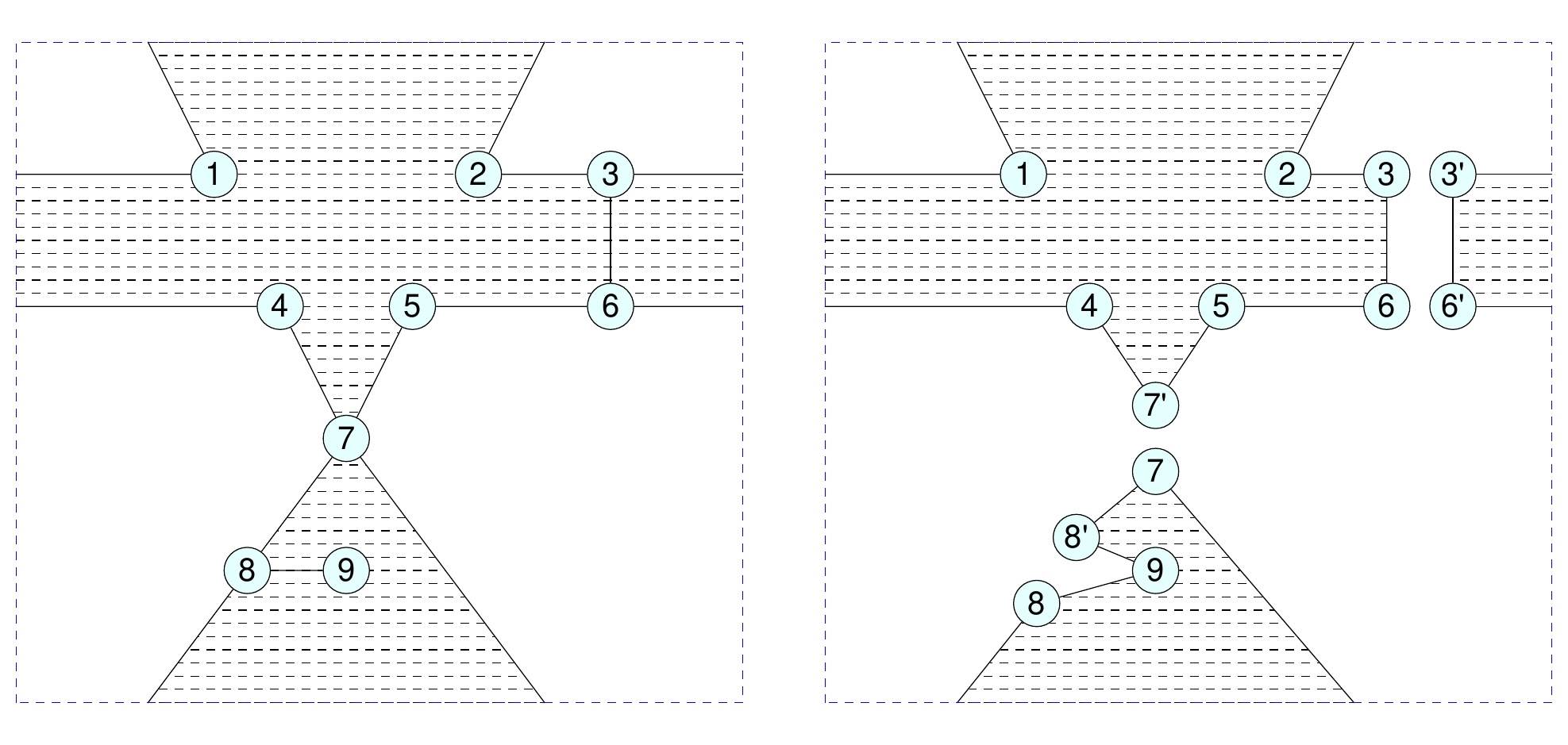}}
 \caption{A 2-cell region in a graph embedded on the torus, {\Large$ \tau$},  before and after vertex- and edge-duplication}
 \label{fig:2cell}
\end{figure}

Here in summary are statistics on 2-cell embeddings of $K_{H(\e)}$. The patterns presented are visible from Table~\ref{tab:heawood} and are easily derived from Euler's formula and the function $I(n)$, given above.

\begin{lem}\label{lem:2cellKn} Let $\e \ge 1$ and suppose $K_{H(\e)}$ has a 2-cell embedding on $S_\e$ (but $S_\e \ne$ {\Large $\kappa$}). Set $i = \Big\lfloor\frac{H(\e)-3}{3}\Big\rfloor$ so that $H(\e) = 3i+3, 3i+4$ or $3i+5$ with $i\ge1$.

\begin{enumerate}
\item\label{e1} If $H(\e) = 3i+3$, then 
$\e = (3i^2-i)/2, (3i^2-i+2)/2, \ldots, $ or $(3i^2+i-2)/2$.
The number of faces of the embedding is given by $f = 3i^2 + 5i + 2, 3i^2 + 5i + 1, \ldots, $ or $3i^2 + 4i + 3$, respectively, and the largest possible face is an $s$-region with $s = 3, 6, \ldots, $ or $3i$, resp.
\item\label{e2} If $H(\e) = 3i+4$, then 
$\e = (3i^2+i)/2, (3i^2+i+2)/2, \ldots, $ or $(3i^2+3i)/2$.
The number of faces of the embedding is given by $f = 3i^2 + 7i + 4, 3i^2 + 7i + 3, \ldots, $ or $3i^2 + 6i + 4$, respectively, and the largest possible face is an $s$-region with $s = 3, 6, \ldots, $ or $3i+3$, resp. 
\item\label{e3} If $H(\e) = 3i+5$, then 
$\e = (3i^2+3i+2)/2, (3i^2+3i+4)/2, \ldots, $ or $(3i^2+5i)/2$.
The number of faces of the embedding is given by $f = 3i^2 + 9i + 6, 3i^2 + 9i + 5, \ldots, $ or $3i^2 + 8i + 7$, respectively, and the largest possible face is an $s$-region with $s = 5, 8, \ldots, $ or $3i+2$, resp. 
 \end{enumerate}
\end{lem}

From the point of view of the genus, given $\e > 0$, we can determine directly whether or not $K_{H(\e)}$ necessarily has a 2-cell embedding on $S_\e$. $K_{H(\e)}$ necessarily has a 2-cell embedding if and only if $\e = (3i^2 - i)/2$ or $(3i^2 + i)/2$ or $(3i^2 + 3i + 2)/2$ for some value of $i > 0$. Thus given $\e > 0$, we compute $H(\e)$ and set $i = \lfloor{H(\e)/3}\rfloor - 1$ so that $H(\e) = 3i+3, 3i+4$, or $3i+5$. Then $K_{H(\e)}$ necessarily embeds with a 2-cell embedding if $I(H(\e)) = \e$; that is, $S_\e$ is the genus surface for $K_{H(\e)}$.

In the results of Table~\ref{tab:heawood} we do not claim that every 2-cell embedding of $K_{H(\e)}$ achieves the maximum face size when that size is greater than three. 
For example when $K_{H(\e)}$ has a largest face being a 5- or 6-region, it might embed as a near-triangulation with one 5- or 6-region, respectively, or it might be a triangulation except for two 4-regions or a triangulation except for a 4- and a 5-region, resp. (An embedding is a \emph{near-triangulation} if at most one region is not 3-sided.)
 
We note from Table~\ref{tab:heawood} and Lemma~\ref{lem:2cellKn} that there are some instances of $\e$ when $K_{H(\e)}$ embeds possibly with an $(H(\e)-1)$-region which might allow for the embedding of two different (not disjoint, but distinct) copies of $K_{H(\e)}$ on $S_\e$, as explained in the next lemma.

\begin{lem}\label{lem:DKn} Let $K_{H(\e)}$ have a 2-cell embedding on $S_\e$, $\e > 0$.
\begin{enumerate}
\item\label{b1} The largest possible face in the embedding is an $(H(\e)-1)$-region. If there is an $(H(\e)-1)$-region, there is just one, and the embedding is a near-triangulation. 
\item\label{b2} If every face of the embedding is at most an $(H(\e)-2)$-region, then no additional copy of $K_{H(\e)}$ can simultaneously embed on $S_\e$.

\item\label{b3} When $K_{H(\e)}$ can embed with an $(H(\e)-1)$-region that is also an\linebreak $(H(\e)-1)$-vertex-region, then two different copies of $K_{H(\e)}$ can embed, by adding a vertex adjacent to all vertices of that region, and then the two complete graphs share a copy of $K_{H(\e)-1}$. Such an embedding is possible only if $H(\e) = 3i+4$ and $\e = (3i^2 + 3i)/2$, and the resulting embedding is a triangulation. 
\end{enumerate}
\end{lem}

We call the latter graph \emph{$DK_{H(\e)}$}; it is also $K_{H(\e)+1}\setminus\{e\}$ for some edge $e$.

\begin{proof} Suppose that $K_{H(\e)}$ has a 2-cell embedding with at least one $s$-region where $s \ge H(\e)-1$. Then Euler's formula plus a count of edges on faces with multiplicities leads to a contradiction to Lemma~\ref{lem:2cellKn} in all cases except when there is precisely one $(H(\e)-1)$-region, $H(\e) = 3i+4$, $\e = (3i^2 + 3i)/2$, and all other faces are 3-regions.

Suppose $K_{H(\e)}$ embeds on $S_\e$ with every face having at most $H(\e)-2$ sides. No two additional vertices in different faces of $K_{H(\e)}$ can be adjacent. For $2 \le k \le 4$, $k$ mutually adjacent, additional vertices cannot form $K_{H(\e)}$ together with $H(\e)-k$ vertices on the boundary of a face.

Proofs of remaining parts follow easily from Euler's Formula and Lemma \ref{lem:2cellKn}. \end{proof}

If $V' \subseteq V(G)$, we denote by $G[V']$ the induced subgraph on the vertices in $V'$; for $E' \subseteq E(G)$, we denote by $G[E']$ the induced subgraph on the edge set $E'$. When $f$ is a face of an embedded $G$, we may also call the subgraph $G[E(f)]$ the boundary of $f$; that is, it may be convenient at times to think of the boundary of a face $f$ as a set $V(f) \cup E(f)$ and at other times as the subgraph $G[E(f)]$.

We restate two very useful corollaries of Thm. 6 in~ \cite{bohme99}. The first involves a case that is not covered in that theorem, but which follows easily from their proof. If $f$ is the infinite face of a connected plane graph, we call the boundary of $f$ the {\em outer boundary} of $G$, and when $G[E(f)]$ is a cycle, we call it the {\em outer cycle}. Without loss of generality we may suppose that for a connected plane graph the outer boundary is a cycle. 

\begin{cor}\label{cor:sixes}(\cite{bohme99}) Let $G$ be a connected plane graph with outer cycle $C$ that is a $k$-cycle with $k \le 6$. If every vertex of $G$ has a list of size at least 6, then a precoloring of $C$ extends to all of $G$ unless $k = 6$, there is a vertex in $V(G) \setminus V(C)$ that is adjacent to all vertices of $V(C)$, and its list consists of six colors that appear on the precolored $C$. 
\end{cor}
 
Then the results of Thm. 6 in~\cite{bohme99} together with Cor.~\ref{cor:sixes} give the next corollary.

\begin{cor}[\cite{bohme99}]\label{cor:bms} Let $G$ be a connected plane graph with outer cycle $C$ that is a $k$-cycle with $3 \le k \le 6$. If every vertex of $G$ has a list of size at least $max(5, k+1)$, then a precoloring of $C$ extends to all of $G$.
\end{cor}

The next lemma is used repeatedly in the proof of Thms.~\ref{thm:manyvs} and~\ref{thm:thebiggie}. It is an extension of the similar result for 5-list-colorings in~\cite{bohme99}. The parameters are motivated by the ``Largest Face'' and $H(\e)$-list sizes from Table~\ref{tab:heawood}. 

\begin{lem}\label{lem:lem1a}
Let $H$ be a connected graph with a 2-cell embedding on $S_\e$, $\e > 0$, and let $f$ be a 2-cell $k$-region of $H$, $k \ge 3$. Let $G$ be a plane graph embedded within $f$ and let $G_f$ be a simple, connected graph that consists of $G$, $H[E(f)]$, and edges joining $V(G)$ and $V(f)$ so that $G_f$ is embedded in the closure of $f$. 
Let $P = \{v_1, \ldots, v_j\}$ be a subset of $V(G_f)$ satisfying $dist(P) \ge 3$. Then if every vertex of $G_f$ has an $\ell$-list except that the vertices of $P$ each have a 1-list, every proper precoloring of $H[E(f)]$ extends to a list-coloring of $G_f$ provided that no vertex of $P$ is adjacent to a vertex of $V(f)$ with the same color as its 1-list, and 
\begin{enumerate}
\item\label{a1} $k = 3$ and $\ell \ge 6$, 
\item\label{a2} $k \ge 4$ and $\ell \ge k+2$, or
\item\label{a3} $k = 6$ or $k \ge 9$, $\ell = k+1$, and there is no vertex $x$ adjacent to $k+1$ vertices of $V(f) \cup \{v_i\}$, for some $i = 1, \ldots, j$, with $x$'s list consisting of $\ell = k+1$ colors that all appear on $V(f) \cup \{v_i\}$.
\end{enumerate}
\end{lem}

\begin{proof} 

Note that $G_f[E(f)] = H[E(f)]$. Also note that the condition\linebreak $dist(P) \ge 3$ guarantees that no vertex of $G_f$ is adjacent to more than one $v_i$. For $v_i \in P \setminus V(f)$, we say that we {\em excise} $v_i$ if we delete it and delete its color from the list of colors for each neighbor that is not precolored. The proof has three cases that together prove parts 1-3 of the lemma.

Case A. Assume $k = 3$ and $\ell \ge 6$, $4 \le k \le 6$ and $\ell \ge k+2$, or $k = 6$ and $\ell = 7$. In these cases first we excise the vertices of $P \setminus V(f)$ so that every remaining vertex of $G$ has a list of size at least 5 for $k = 3$, of size at least $k+1$ for $k = 4, 5, 6$, or else of size at least 6 when $k = 6$.

In the following we may need to do some surgery, perhaps repeatedly, on the face $f$ and its boundary, so that we can apply Cor.~\ref{cor:bms}. First, more easily, when $f$ is a 2-cell $k$-region on which lies no repeated vertex, then $G_f$ is a plane graph with outer cycle a $k$-cycle, $k \le 6$. By Cor.~\ref{cor:bms} a precoloring of $G_f[E(f)]$ extends to $G_f \setminus P$ and this coloring extends to all of $G_f$ unless there is a vertex $x$ with a 6-list, adjacent to six vertices of $V(f)$ with the six colors of $x$'s list. If $x$'s list was decreased to a 6-list, $x$ was adjacent to some vertex $v_i$, but this situation is disallowed by hypothesis in part 3.

Otherwise in a traversal of $W_f$ we visit a vertex more than once and may travel along an edge twice. In the former case, each time we revisit a vertex $x$, we can split that vertex in two, into $x_1$ and $x_2$, and similarly divide the edges incident with $x$ so that the face $f$ is expanded to become the new face $f'$, still a $k$-region, and the graph $G_f$ becomes $G_{f'}$ which is naturally embedded in the closure of $f'$ and contains the same adjacencies. Now there is one more vertex in $V(f')$ and the same set of edges $E(f') = E(f)$ on the boundary and in the boundary subgraph $G_{f'}[E(f')]$. A precoloring of $G_f[E(f)]$ gives a precoloring of $G_{f'}[E(f')]$ in which vertices $x_1$ and $x_2$ receive the same color; we call this procedure {\em vertex-duplication}. In the latter case, when we revisit an edge $e = (y, y')$, we may visit both of its endpoints twice or one endpoint twice and the other just once. We similarly duplicate the edge $e=(y, y')$ by duplicating one or both of its endpoints and splitting $e$ into two new edges $e_1$ and $e_2$. Then we divide the other edges incident with $e$ so that $G_f$ becomes $G_{f'}$ which is naturally embedded in the closure of the new face $f'$, still a $k$-region, but now with one or two more vertices in $V(f')$, the same number of edges in $E(f')$ and in $G_{f'}[E(f')]$, and with one less duplicated edge in $W_{f'}$. A precoloring of $G_f[E(f)]$ gives a precoloring of $G_{f'}[E(f')]$ in which duplicated vertices receive the same color; we call this procedure {\em edge-duplication}. We note that in both duplications there cannot be a vertex $x$ that is adjacent to both copies of a duplicated vertex (since $G_f$ is a simple graph). As an example, the first graph in Fig.~\ref{fig:2cell} 
shows a 2-cell face that is a 13-region, in which vertices 3, 6, 7, and 8, are repeated, and edges 3-6 and 8-9 are repeated. Vertex- and edge-duplication produces the second graph, which has a new face that is a 13-region and whose facial walk is a cycle given by 1-8-9-8$'$-7-2-3-6-5-7$'$-4-6$'$-3$'$-1.

In all cases after vertex- and edge-duplication, the 2-cell $k$-region $f$ becomes a 2-cell $k$-region $f^*$ with no repeated vertex or edge on the outer boundary. $G_f$ has been transformed into a plane graph $G_{f^*}$ with outer cycle, $G_{f^*}[E(f^*)]$, of length $k \le 6$. The precoloring of $G_f[E(f)]$ has become a precoloring of $G_{f^*}[E(f^*)]$ with duplicated vertices receiving the same color. Then by Cor.~\ref{cor:bms}, the precoloring of $G_{f^*}[E(f^*)]$ extends to $G_{f^*} \setminus P$ and so the precoloring of $G_f[E(f)]$ extends to $G_f \setminus P$ and to all of $G_f$ since the exceptional case of part 3 cannot occur. (Since $f^*$ is at most a 6-region and has a duplicated vertex, it is a $t$-vertex-region for some $t < 6$, and there cannot be a vertex adjacent to six vertices of $V(f^*)$.) 

Case B. Suppose $k \ge 7$ and $\ell \ge k+2$ so that in all cases $\ell \ge 9$. For $v \in V(G)$, let $E_f(v)$ denote the set of edges joining $v$ with a vertex of $V(f)$. Suppose there is a vertex $x$ of $V(G)$ that is adjacent to at least $k-3$ vertices of $V(f)$. If $x = v_i$ for some $i, 1 \le i \le j$, then $G_f[E(f) \cup E_f(v_i)]$ can be properly colored by assumption. If $x \ne v_i$ for any $i, 1 \le i \le j$, then $x$ is adjacent to either one or no vertex $v_i$, and since $x$ has an $\ell$-list, $\ell \ge k+2$, the coloring of $G_f[E(f) \cup E_f(v_i)]$ (respectively, $G_f[E(f)]$) extends to $x$. In all cases $G_f[E(f) \cup E_f(x)]$ divides $f $ into regions of size at most 6, and the coloring of $G_f[E(f) \cup E_f(x)]$ extends to the interior of each $s$-region, $3 \le s \le 6$, by Case A since interior vertices, other than the $v_i$, have 9-lists.

Otherwise every vertex $x$ in $G $ is adjacent to at most $k-4$ vertices of $V(f)$. For each such vertex $x$ we delete from its list the colors of $V(f)$ to which it is adjacent. This may reduce the list for $x$ to one of size six or more. Next we excise the vertices of $P$ in $G \setminus V(f)$, resulting in the planar graph $G \setminus P$ with every vertex having a list of size at least five, which can be list-colored by \cite{thomassen94}. This list-coloring is compatible with the precoloring of $G_f[E(f)]$ and extends to $P$ and so to all of $G_f$.

Case C. The case of $k = 6, \ell = 7$ was covered in Case A.
 Suppose that $k \ge 9$ and $\ell = k+1 \ge 10$. Suppose there is a vertex $x$ of $V(G)$ that is adjacent to at least $k-4$ vertices of $V(f)$. As before, if $x = v_i$ for some $i, 1 \le i \le j$, then $G_f[E(f) \cup E_f(v_i)]$ can be properly colored by assumption. If $x \ne v_i$ for any $i, 1 \le i \le j$, then $x$ is adjacent to one or no vertex $v_i$, and the coloring of $G_f[E(f) \cup E_f(v_i)]$ (resp., $G_f[E(f)]$) extends to $x$ in all cases unless (since $\ell = k+1$) $x$ is adjacent to all vertices of $V(f) \cup \{v_i\}$ for some $i, 1 \le i \le j$, and $x$'s list consists of $\ell$ colors all appearing on $V(f) \cup \{v_i\}$. We have disallowed this case. Now $G_f[E(f) \cup E_f(x)]$ forms a graph that consists of triangles and $s$-regions with $s \le 7$. The coloring of $G_f[E(f) \cup E_f(x)]$ extends to the interior of each region by the previous cases, since $\ell \ge 10$.
 
Otherwise every vertex $x$ of $G$ is adjacent to at most $k-5$ vertices of $V(f)$, and we proceed as in the proof of Case B by decreasing the lists of vertices adjacent to $V(f)$ and excising all the $v_i$ to create a planar graph with every vertex having at least a 5-list. The resulting graph is list-colorable with a coloring compatible with that of $G_f[E(f)]$ and extending to $G_f$. \end{proof}
 
\section{Results on $K_n$ genus surfaces}

Most parts of the proof of the next lemma are clear; these results are used repeatedly in the proof of the main results.

\begin{lem}\label{lem:Kn}
\begin{enumerate}

\item\label{c1} Suppose at most one vertex of $K_n$ has a 1-list, at least one vertex has an $n$-list, and the remaining vertices have $(n-1)$-lists or $n$-lists. Then $K_n$ can be list-colored.
\item\label{c2} If one vertex of $DK_n$ has a 1-list and all other vertices have $n$-lists, then $DK_n$ can be list-colored. 
\item \label{c3} If at most six vertices of $DK_n$, $n 
\ge 7$, have lists of size $n-1$ and all others have $n$-lists, then $DK_n$ can be list-colored. 

\end{enumerate}
\end{lem}

\begin{proof} We include the proof of part~\ref{c3}. Suppose that one of the two vertices of degree $n-1$, say $x$, has an $n$-list. Then $K_n$ = $DK_n \setminus \{x\}$ has at most six vertices with $(n-1)$-lists and can be list-colored since $n \ge 7$. This coloring extends to $x$ which has an $n$-list and is adjacent to $n-1$ vertices of the colored $K_n$. Otherwise both vertices of degree $n-1$, say $x$ and $y$, have $(n-1)$-lists, $L(x)$ and $L(y)$ respectively. Suppose there is a common color $c$ in $L(x)$ and $L(y)$. Then coloring $x$ with $c$ extends to a coloring of $K_n = DK_n \setminus \{y\}$ after which $y$ can also be colored with $c$. 
Otherwise $L(x)$ and $L(y)$ are disjoint. Suppose that when $DK_n \setminus \{y\}$ is list-colored, the colors on $K_{n-1} = DK_n \setminus \{x, y\}$ are precisely the $n-1$ colors of $L(y)$ so that the coloring does not extend. If there is some vertex $z$ of $K_{n-1}$ with an $n$-list that contains a color not in $L(y)$ and different from the color $c_x$ used on $x$, we use $c_x$ on $z$, freeing up the previous color of $z$ for $y$. Otherwise, for every $z$ with an $n$-list, that list equals $L(y) \cup \{c_x\}$. Besides these vertices of $K_{n-1}$ with prescribed $n$-lists, there are at most four others in $K_{n-1}$ which have $n-1$ lists. These four vertices might be colored with colors from $L(x)$, but that still leaves at least one color $c_x' \ne c_x$ in $L(x)$ that has not been used. We change the color of $x$ to $c_x'$ and the color of one of the $n$-list vertices of $K_{n-1}$ to $c_x$, thus freeing up that vertex's previous color to be used on $y$.
\end{proof}

\begin{thm}\label{thm:easy}
Suppose $G$ embeds on $S_\e$, $\e > 0$, and does not contain $K_{H(\e)}$. Then when every vertex of $G$ has an $H(\e)$-list except that the $j$ vertices of $P = \{v_1, \ldots, v_j$\}, $j \ge 0$, have 1-lists and $dist(P) \ge 3$, then $G$ is list-colorable.
\end{thm} 

\begin{proof} Let $G$ embed on $S_\e$, $\e > 0$, and suppose $G$ does not contain $K_{H(\e)}$. We excise the vertices of $P = \{v_1, \ldots, v_j\}$, if present, leaving a graph with all vertices having at least $(H(\e)-1)$-lists since $dist(P) \ge 3$. By~\cite{bohme99, kral06}, the smaller graph can be list-colored, and that list-coloring extends to $G$.
\end{proof}

In particular this result holds for all graphs on the Klein bottle since $K_7$ does not embed there. The first value not covered by the next theorem is $\e = 3$ with $H(\e) = 7$. 

\begin{thm}\label{thm:manyvs}

Suppose $G$ has a 2-cell embedding on $S_\e$, $\e > 0$, and contains $K_{H(\e)}$. Then when every vertex of $G$ has an $H(\e)$-list except that the $j$ vertices of $P = \{v_1, \ldots, v_j$\}, $j \ge 0$, have 1-lists, $G$ is list-colorable provided that 
$\e$ is of the form $\e = (3i^2 - i)/2$, $(3i^2 + i)/2$, or $(3i^2 + 3i + 2)/2$, for some $i \ge 1$, and $dist(P) \ge 4$.

\end{thm}

\begin{proof}
We know that $K_{H(\e)}$ necessarily has a 2-cell embedding on $S_\e$ for $\e = 1, 4$ as does $K_7$ on {\Large$ \tau$}. ($K_6$ and $K_7$ may or may not have 2-cell embeddings on {\Large $\kappa$} and on $S_3$, respectively.)

The values $\e = (3i^2 - i)/2$, $(3i^2 + i)/2$, or $(3i^2 + 3i + 2)/2$ for some $i \ge 1$ are those for which $K_{H(\e)}$ necessarily has a 2-cell embedding on $S_\e$; they give the value of the genus surface of $K_{H(\e)}$ for each of the modulo 3 classes of $H(\e)$. Since $dist(P) \ge 4$, at most one vertex $v_k \in P$ is in or is adjacent to a vertex of $K_{H(\e)}$ (but not both), and in the latter case $v_k$ is adjacent to at most $H(\e)-1$ vertices of the complete graph since $K_{H(\e)+1}$ does not embed on $S_\e$. Thus in all cases $K_{H(\e)} \cup P$ can be list-colored by Lemma~\ref{lem:Kn}.\ref{c1}. When $\e = 1$, $H(\e)$ = 6, and $K_6$ embeds as a triangulation on $S_1$. When $\e > 1$, if $\e = (3i^2 - i)/2$ or $(3i^2 + i)/2$, $K_{H(\e)}$ embeds as a triangulation, and if $\e=(3i^2 + 3i + 2)/2$, $K_{H(\e)}$ embeds with the largest face size at most five, and in all cases $H(\e) \ge 7$. Hence we apply Lemma~\ref{lem:lem1a} for $\e \ge 1$ 
to see that the list-coloring of $K_{H(\e)}$ extends to the interior of each of its faces and so $G$ is list-colorable. 
\end{proof}

A similar proof would show that when the orientable surface $S_\e$ with $\e$ even is the orientable genus surface for $K_{H(\e)}$ (i.e., when $\e$ is even and gives the least Euler genus such that $K_{H(\e)}$ embeds on orientable $S_\e$), then for every $G$ with a 2-cell embedding on orientable $S_\e$ and containing $K_{H(\e)}$ the same list-coloring result holds. The first corollary of Section 1 also follows easily.

\begin{proof}[Proof of Cor.~\ref{cor:firstwidth}]
Suppose $H(\e) = 3i+3, i \ge 1$. If $\e = (3i^2 - i)/2$, then $K_{H(\e)}$ embeds with $f = (i+1)(3i+2)$ faces by Lemma~\ref{lem:2cellKn}.\ref{e1}. $K_{H(\e)}$ contains $(3i+3)(3i+2)(3i+1)/6$ 3-cycles, more than the number of faces so that $K_{H(\e)}$ embeds with a noncontractible 3-cycle. Thus in this case $G$ cannot contain $K_{H(\e)}$ and by Thm.~\ref{thm:easy}, $G$ can be list-colored. If $\e = (3i^2 - i + 2)/2$, \ldots, or $(3i^2 + i - 2)/2$, then $K_{H(\e)}$ embeds with fewer than $f = (i+1)(3i+2)$ faces and so the same result holds.

When $H(\e) = 3i+4$ or $3i+5, i \ge 1$, an analogous proof shows that $G$ cannot contains $K_{H(\e)}$ and so is list-colorable. 

To see that distance at least 3 is best possible for the precolored vertices, take a vertex $x$ with a $k$-list $L(x)$ and attach $k$ pendant edges to vertices, precolored with each of the colors of $L(x)$. 
\end{proof}

\section{All surfaces}

First we explore some topology of surfaces and non-2-cell faces of embedded graphs. Cycles on surfaces (i.e., simple closed curves on the surface), for both orientable and nonorientable surfaces, are of three types: contractible and surface-separating, noncontractible and surface-separating, and noncontractible and surface-nonseparating. (When the meaning is clear, we suppress the prefix ``surface.'') A non-2-cell face of an embedded graph must contain a noncontractible surface cycle within its interior. For example, in the second graph in Fig.~\ref{fig:2cell}, the shaded region is a 2-cell face, and the unshaded region is a non-2-cell face that contains a noncontractible and nonseparating cycle. 
(For a more detailed discussion see Chapters 3 and 4 of~\cite{mt}.)

Suppose $f$ is a non-2-cell face of $K_{H(\e)}$ embedded on $S_\e$. We repeatedly ``cut" along simple noncontractible surface cycles that lie wholly within the face $f$ until the ``derived'' face or faces become 2-cells. Each ``cut" is replaced with one or two disks, creating a new surface, and with each ``cut" $K_{H(\e)}$ stays embedded on a surface $S_{\e'}$ with $\e' < \e$. 
Below we explain this surface surgery and count the number of newly created faces, called {\em derived} faces in the surgery.

\begin{lem}\label{lem:2cellface} Suppose $K_{H(\e)}$ embeds on $S_\e$, $\e > 0$. Then the largest possible 2-cell face in the embedding is an $(H(\e)-1)$-region. 
\end{lem}

\begin{proof} Suppose the embedded $K_{H(\e)}$ has a non-2-cell $k$-region $f$; initially there are no derived faces. In $f$ we can find a simple noncontractible cycle $C$, disjoint from its boundary, $V(f) \cup E(f)$. If $C$ is surface-separating, it is necessarily 2-sided. We replace $C$ by two copies of itself, $C$ and $C'$, and insert in each copy a disk, producing surfaces $S(1)$ and $S'(1)$, each with Euler genus that is positive and less than $\e$. Since $K_{H(\e)}$ is connected, it is embedded on one of these surfaces, say $S(1)$. The face $f$ of $K_{H(\e)}$ on $S_\e$ becomes the derived face $f_1$ of $K_{H(\e)}$ on $S(1)$ and retains the same set of boundary vertices $V(f_1) = V(f)$ and edges $E(f_1) = E(f)$ so that $f_1$ is also a $k$-region. Initially $f$ is not a derived face, $f_1$ becomes a derived face and the Euler genus decreases by at least 1. If, later on in the process, $f$ is a derived face, then $f_1$ is also a derived face, the number of derived faces does not increase, and the Euler genus decreases by at least 1. 

If $C$ is not surface-separating and is 2-sided, we duplicate it and sew in two disks, as above, to create one new surface $S(1)$ of lower and positive Euler genus on which $K_{H(\e)}$ is embedded. If $C$ was not separating within the face $f$, then the derived face $f_1$ keeps the same set of boundary vertices and edges as $f$ and remains a $k$-region. As above, the number of derived faces increases by at most 1 and the Euler genus decreases by at least 2. If $C$ was separating within the face $f$, then $f$ splits into two derived faces $f_1$ and $f'_1$. Each vertex of $V(f)$ and each edge of $E(f)$ appears on one of these derived faces or possibly two when it was a repeat on $f$. More precisely, if $f_1$ is a $k_1$-region and $f'_1$ is a $k'_1$-region, then necessarily $k_1 + k'_1 = k$. In this case the Euler genus decreases by 2 and number of derived faces increases by at most 2, increasing by 2 only when the face being cut was an original face of $K_{H(\e)}$. If $C$ is not surface-separating and is 1-sided, we replace $C$ by a cycle $DC$ of twice the length of $C$ and insert a disk within $DC$, producing a surface $S(1)$ with Euler genus that is less than $\e$. $K_{H(\e)}$ remains embedded on $S(1)$, necessarily with positive Euler genus, and the derived face $f_1$ keeps the same boundary vertices and edges as $f$, remaining a $k$-region. Thus the number of derived faces increases by at most 1 and the Euler genus decreases by at least 1.

Now we prove the lemma by induction on the number of non-2-cell faces of the embedded $K_{H(\e)}$. We know the conclusion holds when there are no non-2-cell faces by Lemma~\ref{lem:DKn}. Otherwise let $f$ be a non-2-cell $k$-region. We repeatedly cut along simple noncontractible cycles within $f$ and its derived faces, creating surfaces $S(1)$, $S(2)$, \ldots on which $K_{H(\e)}$ remains embedded. We continue until every derived face of $f$ is a 2-cell. Then $K_{H(\e)}$ is embedded on, say, $S_{\e'}$ with $\e' < \e$ and has fewer non-2-cell faces. By induction each 2-cell face has size at most $H(\e)-1$ and thus every original 2-cell face, which has not been affected by the surgery, also has size at most $H(\e)-1$. 
\end{proof}

We have purposefully proved more within the previous proof.

\begin{cor}\label{cor:cutgenus} Suppose $K_{H(\e)}$ has a non-2-cell embedding on $S_\e$, and suppose that after cutting along noncontractible cycles in non-2-cell faces, $K_{H(\e)}$ has a 2-cell embedding on $S_{\e'}$, $\e' < \e$. Then the number of faces in the latter embedding that are derived from faces in the original embedding is at most $\e - \e'$.
\end{cor}

\begin{proof} In the previous proof we saw that with some cuts the number of derived faces is increased by at most 1 and the Euler genus is decreased by at least 1; let $c_0$ denote the number of cuts in which there is no increase in the number of derived faces and $c_1$ the number of cuts in which there is an increase of 1 in the number of derived faces. If the increase is always at most 1, then the result follows. The number of derived faces is increased by 2 precisely when the cutting cycle $C$ within a face $f'$ is 2-sided, is not surface-separating, is separating within $f'$, and $f'$ is an original face of the embedding. In that case the Euler genus is decreased by 2 also; let $c_2$ denote the number of such cuts. Then the decrease in the Euler genus, $\e -\e'$ is at least $c_0 + c_1 + 2c_2 \ge c_1 + 2c_2$, which equals the number of derived faces.
\end{proof}

\begin{thm}\label{thm:thebiggie} Given $\e > 0$ and $G$ a graph on $n$ vertices that has a 2-cell embedding on $S_\e$, suppose that $G$ contains $K_{H(\e)}$. If $P \subset V(G)$ satisfies $dist(P) \ge 4$, then if the vertices of
$P$ each have a 1-list and every other vertex of $G$ has an $H(\e)$-list, then $G$ can be list-colored. 
\end{thm}

\begin{proof} The proof is by induction on $\e$ and on $n$. We know the theorem holds for $G$ with a 2-cell embedding on $S_\e$ for $1 \le \e \le 2$ by Thm.~\ref{thm:manyvs}. Consider graphs with 2-cell embeddings on $S_{\e^*}$ for $\e^* \ge 3$. For each such embedded graph, the subgraph $K_{H(\e^*)}$ inherits an embedding on $S_{\e^*}$, and $H(\e^*) \ge 7$.

Since $dist(P) \ge 4$ we know that at most one vertex of $P$ lies in or is adjacent to a vertex of $K_{H(\e^*)}$. If there is one, call it $v_i^*$ and if not, ignore reference to $v_i^*$ in the following. By Lemma~\ref{lem:Kn}.\ref{c1} we know that $G[V(K_{H(\e^*)}) \cup \{v_i^*\}]$ can be list-colored since $v_i^*$ is adjacent to at most $H(\e^*)-1$ vertices of $K_{H(\e^*)}$ (because $K_{H(\e^*)+1}$ does not embed on $S_{e^*}$). If $G$ contains a vertex $x$ in neither $V(K_{H(\e^*)})$ nor $P$, then $G[V(K_{H(\e^*)}) \cup \{x\}]$ can be list-colored by
first coloring $K_{H(\e^*)}$ and then coloring $x$, which has an $H(\e^*)$-list and is adjacent to at most $H(\e^*)-1$ vertices of $K_{H(\e^*)}$.

Thus on surface $S_{\e^*}$ we know the result holds for every graph on $n$ vertices with $n \le H(\e^*)+1$. Let $G$ have $n^*$ vertices, $n^* > H(\e^*)+1$, and have a 2-cell embedding on $S_{\e^*}$.

Let $f$ be a $k$-region in the inherited embedding of $K_{H(\e^*)}$ with incident vertices $V(f)$ and edges $E(f)$, and let $G_f$ denote the subgraph of $G$ lying in the closure of $f$, $f \cup V(f) \cup E(f)$. Suppose $f$ is a 2-cell face of $K_{H(\e^*)}$ in whose interior lie vertices of $V(G) \setminus \{V(f) \cup \{v_i^*\}\}$; call these interior vertices $U_f$. Then after deleting the vertices of $U_f$, $G \setminus U_f$ has a 2-cell embedding on $S_{\e^*}$ with fewer than $n^*$ vertices, contains $K_{H(\e^*)}$, and contains vertices of $ P' \subseteq P$ with $dist(P') \ge 4$. By induction $G \setminus U_f$ is list-colorable. By Lemma~\ref{lem:2cellface} $k \le H(\e^*)-1$. We claim that the resulting list-coloring of $G[V(f) \cup \{v_i^*\}]$ extends to $G_f$. 

If $k \le H(\e^*)-2$, then the coloring extends by Lemma~\ref{lem:lem1a}.\ref{a1} and \ref{lem:lem1a}.\ref{a2}. Otherwise $k = H(\e^*)-1$ and the coloring then extends by Lemma~\ref{lem:lem1a}.\ref{a3}, unless there is a vertex $x$ of $G_f$ that has an $H(\e^*)$-list, is adjacent to $v_i^*$, not in $V(f)$, and to all vertices of $V(f)$, and its $H(\e^*)$-list consists of $H(\e^*)$ colors that appear on its neighbors. Then $G[V(K_{H(\e^*)}) \cup \{x\}]$ forms $DK_{H(\e^*)}$, which triangulates $S_{\e^*}$ and does not contain another vertex of $P$ since $dist(P) \ge 4$. Since $v_i^*$ is adjacent to at most three vertices of $DK_{H(\e^*)}$ (the vertices of a\linebreak 3-region), $G[V(DK_{H(\e^*)}) \cup \{v_i^*\}]$ can be list-colored by Lemma~\ref{lem:Kn}.\ref{c3}.
Then the list-coloring extends to the graph in the interior of each 3-region by Lemma~\ref{lem:lem1a}.\ref{a1} since $H(\e^*) \ge 7$.

Thus we can assume that every vertex of $V(G) \setminus \{V(K_{H(\e^*)}) \cup \{v_i^*\}\}$ lies in a non-2-cell region of the embedding of $K_{H(\e^*)}$ on $S_{\e^*}$. We claim there are two vertices of $K_{H(\e^*)}$ that lie only on its 2-cell faces; we prove that below. One of these might lie in $P$ or be adjacent to $v_i^*$, but the other, say $x^*$, has an $H(\e^*)$-list and is adjacent only to vertices of $K_{H(\e^*)}$, precisely $H(\e^*)-1$ of these. 

In that case we consider $G \setminus \{x^*\}$. If $G \setminus \{x^*\}$ does not contain $K_{H(\e^*)}$, it can be list-colored by Thm.~\ref{thm:easy}. Otherwise $G \setminus \{x^*\}$ does contain $K_{H(\e^*)}$. $G \setminus \{x^*\}$ might have a 2-cell embedding on $S_{\e^*}$ or it might not. In the former case, by induction on $n$ it can be list-colored. Suppose that $G \setminus \{x^*\}$ does not have a 2-cell embedding on $S_{\e^*}$. Then the face $f^*$ that was formed by deleting $x^*$ is the one and only non-2-cell face of that embedding since no other face of $G$ has been changed by the deletion of $x^*$.
Then we cut along noncontractible cycles within $f^*$, as described in Lemma~\ref{lem:2cellface}, until every face, derived from $f^*$, is a 2-cell in $G \setminus \{x^*\}$ now embedded on $S_{\e'}$ with $\e' < \e^*$. We have $H(\e') = H(\e^*)$ since $G \setminus \{x^*\}$ contains $K_{H(\e^*)}$. Thus $G \setminus \{x^*\}$ can be list-colored by induction on the Euler genus, and in all cases that coloring extends to $G$ since $x^*$ has a list of size $H(\e^*)$ which is larger than its degree. 

We return to the claim that there are two vertices of $K_{H(\e^*)}$ that lie only on 2-cell faces of its embedding on $S_{\e^*}$, given that every vertex of $V(G) \setminus \{V(K_{H(\e^*)}) \cup \{v_i^*\}\}$ lies in a non-2-cell face of the embedded $K_{H(\e^*)}$. Since the number of vertices of $G$, $n^*$, is greater than $H(\e^*)+1$, there are some non-2-cell faces containing other vertices of $G$. We count the maximum number of vertices of $K_{H(\e^*)}$ that lie on these non-2-cells to show that number is at most $H(\e^*)-2$.

As in Lemma~\ref{lem:2cellface} we repeatedly cut each non-2-cell face of the embedded $K_{H(\e^*)}$ until all remaining faces, the original and the derived, are 2-cells; suppose $K_{H(\e^*)}$ is then embedded on $S_{\e'}$ with $\e' < \e^*$. We know that every vertex originally on a non-2-cell face of $K_{H(\e^*)}$ is represented on at least one derived face and we show below that the total number of vertices on derived faces is at most $H(\e^*)-2$. We also know that $\e' \ge I(H(\e^*))$. Let $n_1 = \e' - I(H(\e^*))$, which is nonnegative, and $n_2 = \e^* - \e'$, which is positive. The variable $n_1$ will determine the face sizes in the 2-cell embedding of $K_{H(\e^*)}$ on $S_{\e'}$ (see Table~\ref{tab:heawood}), and $n_2$ will determine the maximum number of derived faces that have been created.

We consider the modulo 3 class of $H(\e^*)$, and we begin with the case of $H(\e^*) = 3i+4$, $i \ge 1$. We know that $\e^* \in \{(3i^2+i)/2, \ldots, (3i^2+3i)/2\} = \{I(3i+4), \ldots, I(3i+4)+i\}$ 
so that $n_1 + n_2 \le i$ by Lemma~\ref{lem:2cellKn}. By Cor.~\ref{cor:cutgenus} the number of derived faces is at most $n_2$. We can determine the possible face sizes of a 2-cell embedding of $K_{H(\e^*)}$ on $S_{\e'}$ with $\e' = I(3i+4) + n_1$. A 2-cell embedding on $S_{I(3i+4)}$ is necessarily a triangulation. A 2-cell embedding on $S_{I(3i+4)+1}$ consists of triangles except possibly for one 6-region, or triangles plus two faces whose sizes sum to 9, or triangles plus three faces whose sizes sum to 12 (necessarily three 4-regions). More generally when $\e' = I(3i+4) + n_1$, then the embedding might consist of triangles plus one $(3n_1+3)$-region, or triangles plus two faces whose sizes sum to $3n_1+6$, or triangles plus three faces whose sizes sum to $3n_1+9$, etc. And if we choose $n_2$ faces, all the derived faces, the sum of their sizes can be at most $3n_1 + 3n_2 \le 3i < 3i+2 = H(\e^*)-2$.

For $i \ge 1$, the same calculation holds when $H(\e^*) = 3i+3$, and when $H(\e^*) = 3i+5$, a similar count will work. In the latter case we have $n_1 + n_2 \le i-1$, though the face sizes may be slightly larger. A 2-cell embedding of $K_{H(\e^*)}$ on $S_{I(3i+5)}$ may have triangles plus a 5-region or triangles plus two 4-regions. In general a 2-cell embedding of $K_{H(\e^*)}$ on $S_{I(3i+5)+n_1}$ might have triangles plus one $(3n_1+5)$-region or triangles plus two regions whose sizes sum to $3n_1 + 8$, etc. With $n_2$ faces, all the derived faces, their sum of sizes can be at most $3n_1 + 3n_2 + 2 \le 3i-1 < 3i+ 3 = H(\e^*) - 2$. \end{proof}

We now complete the proof our main result, Thm.~\ref{thm:mainresult}.

 \begin{proof}[Proof of Thm.~\ref{thm:mainresult}]
If $G$ has a non-2-cell embedding on $S_\e$ that contains $K_{H(\e)}$, we can perform surgery on the non-2-cell faces, as we did in the proof of Lemma~\ref{lem:2cellface} and Thm.~\ref{thm:thebiggie}, to obtain a 2-cell embedding of $G$ on a surface of Euler genus $\e' < \e$ that still contains $K_{H(\e)}$, and hence $H(\e') = H(\e)$. We can thus apply Thm.~\ref{thm:thebiggie} to $G$ on $S_{\e'}$. This shows that the result holds for every embedding, 2-cell or non-2-cell, and Thm.~\ref{thm:mainresult} follows.
\end{proof}

The distance bound of 4 in Thms.~\ref{thm:mainresult} and \ref{thm:thebiggie} is best possible, for consider $K_{H(\e)}$ with a pendant edge attaching a degree-1 vertex to each vertex of $K_{H(\e)}$. Give each degree-1 vertex the list $\{1\}$ and place that vertex in the set $P$. When every other vertex has an identical $H(\e)$-list that contains 1, the graph is not list-colorable and $dist(P) = 3$.

The second corollary of Section 1 now follows easily.

\begin{proof}[Proof of Cor.~\ref{cor:twosizes}]
Let $f_1, \ldots, f_j$ be the faces with vertices with smaller lists. Add a vertex $x_i$ to $f_i$ and make it adjacent to all vertices of $V(f_i)$. Give each $x_i$ a 1-list $\{\alpha\}$ where $\alpha$ appears in no list of a vertex of $G$, and add $\alpha$ to the list of each vertex of $V(f_i)$, now the neighbors of $x_i$. Then $G \cup \{x_1, \ldots, x_j\}$ can be list-colored by Thm.~\ref{thm:mainresult} since with $P = \{x_1, \ldots, x_j\}$, $dist(P) \ge 4$, and this coloring is a list-coloring of $G$.
\end{proof}

\section{Concluding Questions}

\begin{enumerate}
 \item \v{S}krekovski~\cite{skrek} has shown the extension of Dirac's theorem that if $G$ is embedded on $S_\e$, $\e \ge 5$, $\e \ne 6, 9$, and does not contain $K_{H(\e)-1}$ or $K_{H(\e)-4} + C_5$, then $G$ can be $(H(\e)-2)$-colored. Is the same true for list-coloring?
 \item If $G$ embeds on $S_\e$ and does not contain one of the two graphs of Question 1, if the vertices of one face have at least $(H(\e)-2)$-lists, and if all other vertices have at least $H(\e)$-lists, can $G$ be list-colored? 
 \end{enumerate}

\subsection*{Acknowledgements}
 We wish to thank D. Archdeacon and a referee for helpful comments and Z. Dvo\v{r}\'{a}k and K.-I. Kawarabayashi for information on background material.

\end{document}